\documentclass[11pt,twoside]{amsart}

\usepackage{amssymb,amsmath,amsthm,latexsym}
\usepackage{mathrsfs}
\usepackage{url}
\usepackage{amsfonts}
\usepackage{graphicx}
\usepackage{hyperref}

\usepackage{graphicx}
\usepackage{amssymb,amsthm,amsmath}

\newtheorem{theorem}{Theorem}[section]

\newtheorem{lemma} [theorem]{Lemma}

\usepackage{enumerate}

\vspace{5cm}
  
\begin{document}

\title{A Note on Deaconescu's Conjecture} 
\author[Sagar Mandal]{Sagar Mandal}
\address{Department of Mathematics and Statistics, Indian Institute of Technology Kanpur\\ Kalyanpur, Kanpur, Uttar Pradesh 208016, India}
\email{sagarmandal31415@gmail.com}

\maketitle
\let\thefootnote\relax
\let\thefootnote\relax
\maketitle

\begin{abstract}
Hasanalizade \cite{new2022} studied Deaconescu's conjecture  for positive composite integer $n$. A positive composite integer $n\geq4$ is said to be a Deaconescu number if $S_2(n)\mid \phi(n)-1$. In this paper, we improve Hasanalizade's result by proving that a Deaconescu number $n$ must have at least seventeen distinct prime divisors, i.e., $\omega(n)\geq 17$ and must be strictly larger than $5.86\cdot10^{22}$. Further, we prove that if any Deaconescu number $n$ has all prime divisors greater than or equal to $11$, then $\omega(n)\geq p^{*}$, where $p^{*}$ is the smallest prime divisor of $n$ and if $n\in D_3$ then all the prime divisors of $n$ must be congruent to $2$ modulo $3$ and $\omega(n)\geq 48$.
\end{abstract}

\section{Introduction}
Lehmer \cite{Lehmer1932} conjectured that if
$\phi(n)\mid (n-1)$ (where $\phi(n)$ is the Euler’s totient function), then $n$ must be a prime number. A positive composite integer that satisfies the condition $\phi(n)\mid (n-1)$ is called a Lehmer number. Let $S_2(n)$ denote Schemmel’s totient function which is a multiplicative function defined by
$$
S_2(p^k) :=
\begin{cases}
0, & \text{if } p = 2 \\
p^{k-1} (p - 2), & \text{if } p > 2
\end{cases}
$$
where $p$ is a prime number and $k\in \mathbb{Z}^{+}$. Deaconescu \cite{Deacon2000} probably inspired by Lehmer considered Schemmel’s totient function and conjectured that for $n\geq 2$
$$S_2(n)\mid \phi(n)-1$$
if and only if $n$ is prime. A positive composite integer $n\geq4$ is said to be a Deaconescu number if $S_2(n)\mid \phi(n)-1$. Hasanalizade \cite{new2022} proved that a Deaconescu number $n$ must be an odd, squarefree positive integer such that $\omega(n)\geq 7$. Further, Hasanalizade provided an upper bound for $n$ using Nielsen's Lemma (\cite{Nielsen2015}, Lemma 1.4) and also proved that for a monic non-constant polynomial $P(X)\in \mathbb{Z}[X]$, there
are at most finitely many composite integers $n$ such that $S_2(n)\mid (\phi(n)-1)$ and $P(S_2(n))\equiv 0 \pmod{\phi(n)}$, an analogous result due to Hernandez and Luca (\cite{Hernandez2008}, Theorem 1).\\
Observe that, Deaconescu's conjecture says that for every $M\geq1$ the set $D_{M}$ (we are using notations used by Hasanalizade) of integers satisfying $MS_2(n)=\phi(n)-1$
contains only prime numbers. Hasanalizade proved that the conjecture is true for $M=1$, that is, the set $D_1$ contains only prime numbers (see \cite{new2022}, Lemma 1). Therefore, we consider $M\geq3$.\\
The following theorem describes the nature of a positive composite integer $n$ which belongs to $D_{3}$. 

\begin{theorem}\label{thm1}
 If $n$ is a Deaconescu number such that $n\in D_3$ then all the prime divisors of $n$ must be congruent to $2$ modulo $3$ and $\omega(n)\geq 48$.
\end{theorem}
We improve the lower bound for $\omega(n)$ for a Deaconescu number $n$ by proving the following theorem.
\begin{theorem}\label{thm2}
If $n$ is a Deaconescu number, then the number of distinct prime divisors of $n$ is at least seventeen, that is, $\omega(n)\geq 17$.    
\end{theorem}
As a consequence of Theorem \ref{thm2}, we can produce a lower bound for a Deaconescu number $n$.
\begin{theorem}\label{thm3}
A Deaconescu number must be strictly larger than $5.86\cdot 10^{22}$. 
\end{theorem}

Throughout this paper, we use $p_{1},p_{2},\dots,p_{\omega(n)}$ to denote prime numbers.
\section{Preliminaries}
In this section, we prove some useful lemmas concerning Deaconescu numbers.
\begin{lemma}\label{l1}
If $n$ is a Deaconescu number having all prime divisors greater than or equal to $11$, then $\omega(n)\geq p^{*}$, where $p^{*}$ is the smallest prime divisor of $n$.
\end{lemma}
\begin{proof}
Let us consider a function $\tau: [7,\infty)\to \mathbb{R}$ defined by $\tau(t)=(1+\frac{1}{t})^{1+t}$. Then
$$\tau'(t)=(1+\frac{1}{t})^{1+t}\big\{ \log(1+\frac{1}{t})-\frac{1}{t}\big\}<0~~~\text{for}~~t\geq7,$$
therefore $\tau$ is a strictly decreasing function of $t$ in $[7,\infty)$ and thus for all $t\in [7,\infty)$
\begin{align}\label{e1}
\tau(t)=(1+\frac{1}{t})^{1+t}\leq (1+\frac{1}{7})^{8}=\tau(7)<3.
\end{align}
Let $n$ be a Deaconescu number with all prime divisors greater than or equal to $11$ and also let $p^{*}$ be the smallest prime divisor of $n$. Then
\begin{align*}
3\leq M=\frac{\phi(n)-1}{S_2(n)}<\frac{\phi(n)}{S_2(n)}&=\prod_{p\mid n}\big(\frac{p-1}{p-2}\big)\\&=\prod_{p\mid n}\big( 1+\frac{1}{p-2}\big)\\
&< \prod_{p\mid n}\big(1+\frac{1}{p^{*}-2}\big)=\big(1+\frac{1}{p^{*}-2}\big)^{\omega(n)}.
\end{align*}
If possible, suppose that $\omega(n)<p^{*}$, then setting $t=p^{*}-2$ in (\ref{e1}) we get
$$\big(1+\frac{1}{p^{*}-2}\big)^{\omega(n)}\leq \big(1+\frac{1}{p^{*}-2}\big)^{p^{*}-1}<\tau(7)<3,$$
it follows that $3\leq M<3$ which is a contradiction. Hence $\omega(n)\geq p^{*}$.
\end{proof}
\begin{lemma}\label{l2}
If $n$ is a Deaconescu number. Then no prime divisor of $n$ is congruent to $1$ modulo $M$.
\end{lemma}
\begin{proof}
Let $p$ be a prime which divides $n$ and suppose that $p\equiv 1 \pmod{M}$. Note that since $\phi(n)\equiv 1 \pmod{M}$ and $(p,n/p)=1$, it follows that
$$0\equiv \phi(p)\cdot \phi(n/p)=\phi(n)\equiv 1 \pmod{M},$$
a contradiction.
\end{proof}
\section{Proof of Theorem \ref{thm1}}
Let $n$ be a Deaconescu number such that $n\in D_3$. Then, from Lemma \ref{l2} we have $p\not\equiv 1 \pmod{M}$ for all prime divisors $p$ of $n$, where $M=\frac{\phi(n)-1}{S_2(n)}\geq 3$, thus in particular if $M=3$ then the prime divisors of $n$ belong to the set $\mathcal{P}=\{3,p: p \text{ is a prime ,~~} p\equiv2 \pmod{3}\}$. If possible, assume that $3\mid n$. Then 
$$n=3\cdot\prod_{i=2}^{\omega(n)}p_i,$$
where $p_{i}\geq 5$. Since $M=3$, $p_{i}=3k_{i}+2$, for some $k_{i}\in \mathbb{Z}^{+}$, for all $2\leq i\leq \omega(n)$. Therefore
$$3\cdot1\cdot \prod_{i=2}^{\omega(n)}\big(3k_i\big)=2\cdot\prod_{i=2}^{\omega(n)}\big(3k_{i}+1\big)-1$$
then
$$2\cdot\prod_{i=2}^{\omega(n)}\big(3k_{i}+1\big)\equiv 1 \pmod{3}$$
 that is
$$2\equiv 1 \pmod{3}$$
but this is impossible. Therefore $3\nmid n$.\\
Now suppose that $\omega(n)\leq 47$, then since $M=3$, if we write
$$n=\prod_{i=1}^{\omega(n)}p_i$$
then $p_{i}=3k_{i}+2$, for some $k_{i}\in \mathbb{Z}^{+}$, for all $1\leq i\leq \omega(n)$. Therefore
$$3\cdot \prod_{i=1}^{\omega(n)}\big(3k_i\big)=\prod_{i=1}^{\omega(n)}\big(3k_{i}+1\big)-1$$
which implies that
\begin{align}\label{3}
3=\prod_{i=1}^{\omega(n)}\big(1+\frac{1}{3k_{i}}\big)-\prod_{i=1}^{\omega(n)}\big(3k_{i}\big)^{-1}.    
\end{align}
Note that
\begin{align*}
\prod_{i=1}^{\omega(n)}\big(1+\frac{1}{3k_{i}}\big)-\prod_{i=1}^{\omega(n)}\big(3k_{i}\big)^{-1}&<\prod_{i=1}^{\omega(n)}\big(1+\frac{1}{3k_{i}}\big)\\
&\leq \prod_{i=1}^{47}\big(1+\frac{1}{3k_{i}}\big)\\
&< \prod_{\substack{k \in \mathcal{A}\\\mathcal{A}=\{k:~5\leq 3k+2\leq 71 \text{~~is prime}\}\\|\mathcal{A}|=10}}\big(1+\frac{1}{3k}\big)\cdot\big(1+\frac{1}{81}\big)^{37}<3
\end{align*}
which clearly contradicts (\ref{3}). Therefore, we must have $\omega(n)\geq 48$. This completes the proof.\qed
\section{Proof of Theorem \ref{thm2}}
 Assume that $n=p_1p_2\cdots p_{\omega(n)}$ with $2<p_1<p_2<\dots<p_{\omega(n)}$ is a Deaconescu number and $\omega(n)\leq 16$. Let $\mathcal{A}=\{p~:~p \text{~is a prime,~} 3\leq p\leq 59 \}$ and $\mathcal{A}^*=\{p~:~p\text{~is a prime,~}p\not\equiv 1 \pmod{5},\text{~and~} 3\leq p\leq 79\}$. Note that $\mathcal{A}$ and $\mathcal{A}^*$ each have the $16$ primes in the appropriate intervals. Now we see that
 $$3\leq M=\frac{\phi(n)-1}{S_2(n)}<\frac{\phi(n)}{S_2(n)}=\prod_{i=1}^{\omega(n)}\big(\frac{p_i-1}{p_i-2}\big)=\prod_{i=1}^{\omega(n)}\big(1+\frac{1}{p_i-2}\big)\leq \prod_{p\in \mathcal{A}}\big(1+\frac{1}{p-2}\big)<6.$$
 It follows that $M=3$ or $M=5$, if $M=5$ then we have
 $$5=M=\frac{\phi(n)-1}{S_2(n)}<\frac{\phi(n)}{S_2(n)}=\prod_{i=1}^{\omega(n)}\big(\frac{p_i-1}{p_i-2}\big)=\prod_{i=1}^{\omega(n)}\big(1+\frac{1}{p_i-2}\big)\leq\prod_{p\in \mathcal{A}^*}\big(1+\frac{1}{p-2}\big)<5,$$
 a contradiction. It follows that $M=3$ and thus, $n\in D_3$. But in this case from Theorem \ref{thm1}, we must have $\omega(n)\geq 48$, again a contradiction. Therefore, our assumption that $\omega(n)\leq 16$ is wrong, and hence we must have $\omega(n)\geq 17$.
\qed
\section{Proof of Theorem \ref{thm3}}
Let $n$ be a Deaconescu number. Then the number of distinct prime divisors of $n$ is greater than or equal to $17$ by Theorem \ref{thm2}. Therefore
$$n\geq \prod_{j=1}^{\omega(n)}P_{j+1},$$
where $P_j$ is the $j$-th prime number (e.g. $P_2=3,P_5=11$). Note that,
$$n\geq \prod_{j=1}^{\omega(n)}P_{j+1}\geq \prod_{j=1}^{17}P_{j+1}>5.86\cdot 10^{22}.$$
This completes the proof.\qed

\end{document}